\documentclass{ws-m3as}

\def \F {{\mathcal{F}_{\text{drag}}}}
\def\e {{\bf e}}
\def \x {{\bf x}}
\def \d {{\bf d}}
\def \n {{\bf n}}

\begin{document}

\markboth{L. Berlyand at el.}{Boundary conditions in modeling of collisions}

%%%%%%%%%%%%%%%%%%% Publisher's Area please ignore %%%%%%%%%%%%%%%%%%%%%%%
%
\catchline{}{}{}{}{}
%
%%%%%%%%%%%%%%%%%%%%%%%%%%%%%%%%%%%%%%%%%%%%%%%%%%%%%%%%%%%%%%%%%%%%%%%%%%

\title{BOUNDARY CONDITIONS IN PDE MODEL OF COLLISIONS OF SWIMMERS}

\author{LEONID BERLYAND}

\address{Department of Mathematics, Pennsylvania State University, \\
University Park, Pennsylvania 16802, USA \\
berlyand@math.psu.edu}

\author{VITALIY GYRYA}

\address{Los Alamos National Laboratory, MS B284\\
Los Alamos, New Mexico 87545, USA\\
vitalishe@gmail.com}

\author{MYKHAILO POTOMKIN}

\address{Department of Mathematics, Pennsylvania State University, \\
University Park, Pennsylvania 16802, USA \\
potomkin@math.psu.edu}

\maketitle

\begin{history}
\received{(Day Month Year)}
\revised{(Day Month Year)}
%\accepted{(Day Month Year)}
\comby{(xxxxxxxxxx)}
\end{history}

\begin{abstract}
The goal of the paper is to determine  boundary conditions in PDE models of collisions of microswimmers in a viscous fluid. We consider two self-propelled spheres (microswimmers) moving towards each other in viscous fluid.
We first show that under commonly used no-slip boundary conditions on the fluid-solid interface the microswimmers do not collide which is a generalization of the well-known no-collision paradox for solid bodies (with no self-propulsion) in a viscous fluid.
Secondly, we show that the microswimmers do collide when the no-slip boundary conditions are replaced by the Navier boundary conditions which therefore provides an adequate model of microswimmers such as swimming bacteria.
The self-propulsion mechanism generates a drag force pulling a bacterium backwards and the collision problem is reduced to the analysis of competition between the drag and self-propulsion.
For no-slip this is done by utilizing the Lorentz Reciprocal Theorem and the analytical solution for two solid spheres in the fluid.
The analysis for the Navier boundary conditions is based on the variational formulation of Stokes problem. A Poincare type inequality for symmetrized gradient is introduced in this work.
\end{abstract}

\keywords{Self-propulsion, Stokes flow, no collision paradox, Navier boundary conditions}

\ccode{AMS Subject Classification: 76D07,76D08}

\section{Introduction}

In recent years active suspensions became a focus of attention in both
biophysical and mathematical communities.
This was motivated by a number of experiments
highlighting differences between active and passive suspensions and
suggesting novel engineering applications.
These experiments (just mention a few) include:
decrease of effective viscosity for active suspensions
compared with passive ones with the same volume fraction of inclusions \cite{SokAra09},
swimming bacteria rotating small asymmetric gears \cite{SokApoGrzAra10},
enhanced mixing in active suspensions of swimming bacteria \cite{SokGolFelAra09},
collective behavior of swimming microorganisms E. coli
observed \cite{WuLib00}
through correlations in swimmer orientations on scales
%(10-20$\mu$m)
order of magnitude larger than the size of individual swimmer.
%(up to 2$\mu$m long).
We refer also to theoretical works \cite{SaiShe12,SaiShe07,HaiSokAraBerKar09}.
Swimmer-swimmer interactions are essential for all of the above phenomena.
Therefore, for the mathematical models to recover the physically observed phenomena
it is important to model accurately these interactions.

Our focus will be on one of the aspects of swimmer-swimmer interactions -- collisions.
It is hard to perform an experiment with collisions of two swimmers. Instead, collisions between swimmers and walls have been
clearly observed in physical experiments \cite{AraSokKesGol07,CisKesGanGol11,DreDunCisGan11}.
As we show in this paper,
some of the commonly accepted models of active swimmers in the fluid considered in literature
do not allow for collisions in finite time.
In this paper we propose a model that captures finite time collisions.

%Understanding of theoretical assumptions that allow for
%collisions, and hence create a better match between theory and experiment,
%is important for the development of the theory of active suspensions.

Modeling of active suspensions consists of several components:
fluid motion, particle (swimmer) motion, particle-fluid interactions, and self-propulsion.
Modeling of fluid motion at the scale of microswimmers (micron scale)
by Stokes equation has been universally accepted.
Particles are commonly taken to be rigid bodies.
The interactions between the particle and a fluid are prescribed through the boundary conditions
on the surface of the particle.
Existing models of self-propelled swimmers use no-slip boundary conditions on the entire swimmer's surface or its part\cite{SaiShe12,GyrLipAraBer11,KanSheTor10,IshSimPed06},
which mean that fluid sticks to the surface and moves with the same velocity as the boundary.
Here we model self-propulsion by a point source term with given magnitude and direction applied at some distance behind the body\cite{HaiSokAraBerKar09,GyrAraBerKar10}.
%Among the different models of propulsion mechanisms we consider
%a propulsion force applied to the fluid on some distance behind the body
%coupled with the fluid reaction force propelling the particle
We present the model in details in Section~\ref{sect:model of microswimmers}.

For passive suspensions (no self-propulsion) there is a well-known no-collision paradox \cite{CoxBre67}.
It states that two rigid particles with no-slip boundary conditions,
immersed in a Stokesian fluid and
pushed towards each other with a constant force will take infinite time to collide.
This is a consequence of an asymptotic relationship $v\sim h$
between the velocity $v$ of the spheres and the distance $2h$ between
the spheres in the limit of small $h$, which leads to asymptotically exponential decay of the distance between the particles.
It is called a paradox due to a mismatch between the theoretical predictions
of no collisions and the experimental observations of possibility of collisions.

A possible remedy to the no-collision paradox
is to replace the no-slip boundary conditions by Navier ones.
The Navier boundary conditions allow for a slip on the fluid/solid interface
with friction linearly proportional to the slip velocity.
The proportionality coefficient $\beta$
can be related to the surface roughness. Navier boundary conditions were derived as a rigorous homogenization limit of models of immersed rigid bodies with mean roughness (size of bumps) of order $\beta$ when $\beta$ is small (see \cite{JagMik03}).
Two passive spheres with Navier conditions on the boundary
pushed towards each other by a constant force will collide in finite time.
This follows from an estimate on the drag force for a two spheres moving towards each other with a unit velocity, which was recently proven in \cite{GerHil11}.

The no collision paradox also holds for active suspensions in the following sense.
Two swimmers moving towards each other in Stokes fluid do not collide in finite time,
%In the present work we prove that,
%analogously to the problem of passive suspensions,
%two swimmers
%with no-slip conditions on the boundary
%moving towards each other take infinite time to collide
see Section~\ref{sect:main results}.
We also prove that for sufficiently large initial velocities
two swimmers with Navier boundary conditions,
moving towards each other, collide in finite time, see Section~\ref{sect:main results}.
The no-slip boundary conditions have been successfully used in the study
of dilute suspensions of noninteracting swimmers\cite{HaiAraBerKar12,HaiSokAraBerKar09,HaiAraBerKar08},
pairwise interacting swimmers\cite{KanSheTor10,GyrAraBerKar10,HerStoGra05}
and collective behavior of swimmers\cite{GyrLipAraBer11,CisKesGanGol11}.

To recover collisions for more concentrated suspensions
(where interactions are important)
the no-slip conditions should be replaced by Navier boundary conditions
(or, possibly by some other conditions).
%
%Therefore, Navier boundary conditions are a better alternative
%to commonly used no-slip conditions.

The key ingredients of our analysis are estimates for the drag force generated
by the propulsion mechanism onto the body of the swimmer.
In the analysis of the model with no-slip boundary conditions we use the Lorentz reciprocal theorem (or the second Green's formula for the Stokes equation, see Theorem \ref{thm: LRT})
to reduce the original problem with singular $\delta$-functions to an auxiliary problem without singularities.
For the reduced problem the exact solution is known in a form of an infinite series.
The analysis of the model with the Navier boundary conditions is based on the variational principle for nonhomogeneous Stokes problem, the representation of the drag force via minimization principle (see Proposition \ref{prop: variation principle}),
and the special Poincar\'{e}-type inequality (Proposition~\ref{prop: poincare special})
are used.
%In fact,
%models with no-slip boundary conditions can be considered as
%a first order approximation ($\beta=0$)
%to the models with Navier condition
%in the limit of small $\beta$.
%{\bf ??? underline line is removed since it is not about the result. it fits more where we introduce the Navier bc.???}

The paper is organized as follows.
The models for swimmers with no-slip and Navier boundary conditions
are presented in Section \ref{sect:model of microswimmers}.
In Section~\ref{sect:review of passive suspensions} we outline some relevant results
for passive spheres.
The main results
(no collisions for no-slip boundary conditions, finite time collisions for Navier boundary conditions) are presented in Section~\ref{sect:main results}.
The proofs of these results are presented in Sections ~\ref{sect:proof of the noslip result} and \ref{sect:proof of the Navier result}, respectively.

\section {PDE/ODE model for microswimmers}
\label{sect:model of microswimmers}

We consider the motion of two self-propelled spheres (swimmers).
At time $t\geq 0$ they are separated by distance $2h(t)$ and occupy regions $B_1=\left\{\x:|\x-\x_c^1|<1\right\}$ and $B_2=\left\{\x:|\x-\x_c^2|<1\right\}$,
see Fig 1.%\ref{fig1}.
Here  $\x_c^1=(0,0,-1-h(t))$ and $\x_c^2=(0,0,1+h(t))$ are centers of swimmers' bodies.
Assume that the swimmers are immersed in a viscous incompressible fluid.
The fluid occupies the domain $\Omega_{h(t)}=\mathbb R^3\backslash \overline{(B_1\cup B_2)}$. For a given distance $h>0$ the fluid is described by a vector field $u$
which solves the Stokes problem in $\Omega_h$:
\begin{equation}\label{stokes}
    -\Delta u+\nabla p=\sum\limits_{i=1,2}\delta(\x-\x_p^i)f_p\d^i,\;\;\nabla\cdot u=0.
\end{equation}
Here $\delta(\x)$ is a delta-function,
$f_p$ is a positive parameter which represents the intensity of self-propulsion,
$\d^1=(0,0,1)$ and $\d^2=(0,0,-1)$ are directions of the swimmers' motion, $\x_p^1=(0,0,-1-h-\lambda)$ and $\x_p^2=(0,0,1+h+\lambda)$ are ends of bacteria flagella and $\lambda$ is the length of flagella. 
We also assume that $\nabla (u-u_{\Phi})\in \left[L^2(\Omega_h)\right]^3$,
where $u_{\Phi}(\x)=\mathcal{G}(\x-\x^1_p)\d^1+\mathcal{G}(\x-\x^2_p)\d^2$ and $\mathcal{G}(\x)=1/8\pi(1/|\x|\cdot\mathrm{I}+\x\x/|\x|^3)$ is the Oseen tensor.

\begin{figure}{}\label{fig1}
\vskip1mm\centerline{\includegraphics[height=3cm]{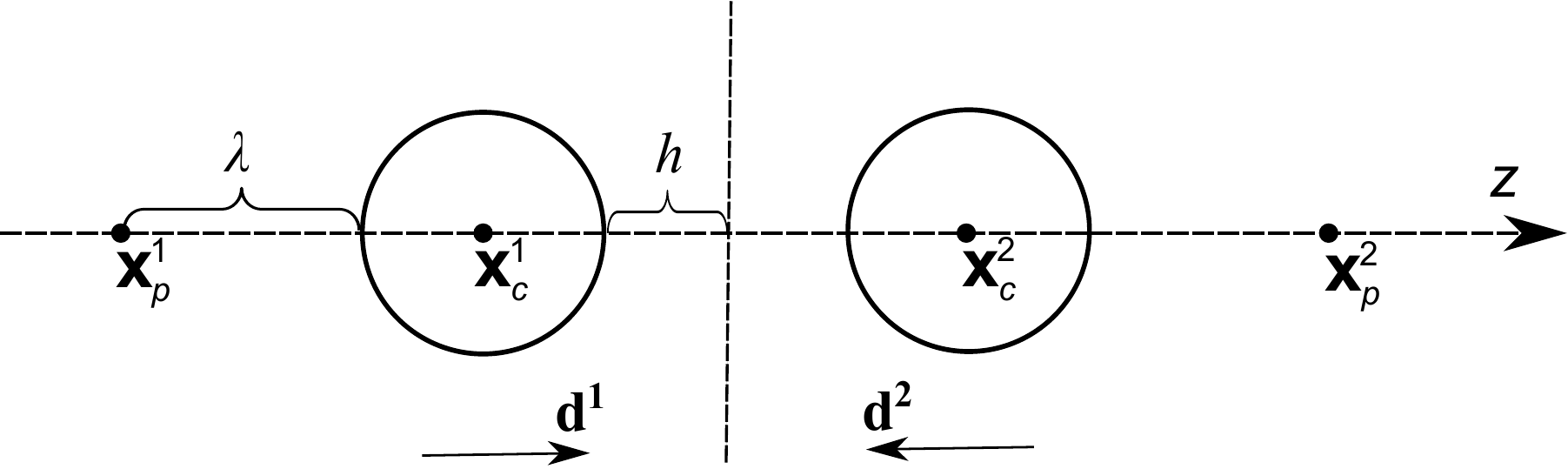}}\caption{Two symmetric swimmers on distance $h$ and with flagellum length $\lambda$}
\end{figure}
Equations \eqref{stokes} are supplemented by one of two following boundary conditions on $\partial B_i$:
\begin{eqnarray}
    &&  u+h'(t)\d^i=0; \label{noslip}\\
    && (u+h'(t)\d^i)\cdot \n=0,
    \;\;(u+h'(t)\d^i)\times \n=-2\beta[D(u)\n]\times \n.
    \label{navier}
\end{eqnarray}
The boundary conditions \eqref{noslip} and \eqref{navier} are no-slip and Navier boundary conditions, respectively. Here $\beta$ is a positive parameter, and $D(u)=\frac{1}{2}(\nabla u+(\nabla u)^{T})$. Note that if $\beta=0$, then \eqref{navier} becomes \eqref{noslip}. The vector $u+h'(t)\d^i$ is the relative velocity of fluid particles near a swimmer with respect to this swimmer. Navier boundary conditions \eqref{navier} state that tangential components of relative velocity and the stress on the boundary are proportional.

For given flow $u$ and pressure $p$ the drag force on sphere $\partial B_i$ is
\begin{equation}\label{stress}
    \F(u)=-\int\limits_{\partial B_i}\sigma(u,p)\n ds\cdot\d^i,
\end{equation}
where $\sigma(u,p)=-2D(u)+p\mathrm{I}$ is the stress tensor and
$\n$ is a normal vector on $\partial B_i$.
%inward for $B_i$
Due to the symmetry of the problem
(spheres are identical) the right hand side of \eqref{stress}
does not depend on $i=1,2$.
The force balance equation is similar for both spheres and looks as follows
\begin{equation}\label{balance}
    mh''(t)-\F(u)+f_p=0,
\end{equation}
where $m\geq 0$ is mass of a swimmer.

We assume that the swimmers do not rotate and, thus, torque balance equation trivially holds.

\begin{remark} \label{remark1}
In the balance equation \eqref{balance} the propulsion enters explicitly in the third term, $f_p$, and implicitly in the second term, $\F(u)$.
We wish to rewrite \eqref{balance} in such a way that the propulsion enters only in one of the terms and does so explicitly.
We can consider the drag force $\F(u)$ as a function of distance $h(t)$ and speed $h'(t)$ in the following sense.
For a given distance $h(t)$ and velocity $h'(t)$ the function $u$ is uniquely determined by the Stokes problem \eqref{stokes} with boundary conditions \eqref{noslip} or \eqref{navier}.
Next our goal will be to explicitly write out the dependence of $\F(u)$ on $h(t)$ and $h'(t)$.
Due to the linearity of the Stokes problem we can decompose the fluid flow $u$ into two components
\begin{equation}
    u=-h'(t)\tilde v+f_p \tilde w,
\end{equation}
where the first component $-h'(t)\tilde v$ is generated by the motion of the spheres and
the second component $f_p \tilde w$ is due to the propulsion force.
The function $\tilde v$ and $\tilde{w}$ solve the following Stokes problems
\[
    \left\{
    \begin{array}{l}
    \Delta \tilde{v}-\nabla p_{\tilde v}=0\text{ in }\Omega_{h(t)},\\
    \nabla\cdot \tilde v=0\text{ in }\Omega_{h(t)},\\
    (\tilde v-\d^i)\cdot \n=0\text{ on }\partial B_i,\; i=1,2,\\
    (\tilde v-\d^i)\times \n=-2\beta\left[D(\tilde v)\n\right]\times \n,
    \end{array}
    \right.
    \qquad\left\{
    \begin{array}{l}
    \Delta \tilde w-\nabla p_{\tilde w}=\sum\limits_{i=1,2}\delta(\x-\x_i^p)\d^i\text{ in }\Omega_{h(t)},\\
    \nabla\cdot \tilde w=0\text{ in }\Omega_{h(t)},\\
    \tilde w\cdot \n=0\text{ on }\partial B_i,\; i=1,2,\\
    \tilde w\times \n=-2\beta\left[D(\tilde w)\n\right]\times \n.
    \end{array}
    \right.
\]
Due to the linear dependence of the drag force $\F$ on the flow $u$ we obtain
\begin{equation*}
\F(u)=-h'(t)\kappa_{\text{pass}}+f_p \kappa_{\text{prop}},
\end{equation*}
where $\kappa_{\text{pass}}:=\F(\tilde v)$ 
is the drag force on a passive sphere moving with a unit velocity, and
$\kappa_{\text{prop}}:=\F(\tilde w)$ 
is the drag force on a motionless sphere due to a unit propulsion force.
Note that both coefficients $\kappa_{\text{prop}}$ and $\kappa_{\text{pass}}$
depend on the distance $h$.

Using the drag coefficients $\kappa_{\text{prop}}$ and $\kappa_{\text{pass}}$,
we can rewrite the balance equation \eqref{balance} as follows:
\begin{equation}\label{expl_be}
    m h'' (t)+\kappa_{\text{pass}}h'(t)+f_p(1-\kappa_{\text{prop}})=0.
\end{equation}

Now, the self-propulsion parameter enters only the third term in \eqref{expl_be},
unlike \eqref{balance} where $\F(u)$ depends on $f_p$.
If $f_p=0$, then \eqref{expl_be} becomes a force balance equation for passive spheres.

%The form of balance equation \eqref{expl_be} explicitly shows the contribution of self-propuslion.  If one equates the third term, $f_p(1-\kappa_{\text{prop}})$, to zero, he will obtain the balance equation for passive spheres (swimmers with the zero self-propulsion). Note that propulsion mechanism creates two forces. On the one hand, it pushes a swimmer in front with strength $f_p$ and, on the other hand, it repels the fluid creating a Stokeslet with the same strength $f_p$ and the direction opposite to the direction of the swimmer's motion.
\end{remark}

\section{Review of known results about collision for passive inclusions}
\label{sect:review of passive suspensions}

We now briefly review the result for passive spheres.

Consider two unit spheres moving along towards each other along the common axis $z$ 
(see Fig 1.).%\ref{fig1}).
Let $h(t)$ be a half-distance between the spheres. 
Then the speed of each sphere is $-h'(t)$.
Assume that the external force $f_{\mathrm{ext}}$ pushes the spheres towards each other. 
Denote by $\F$ a magnitude of the drag force of the fluid computed by the formula \eqref{stress}. 
The force balance on each sphere is
\begin{equation}\label{be}
-\F+f_{\mathrm{ext}}=0.
\end{equation}

From remark \ref{remark1} the drag force $\F$ depends linearly on the speed of the spheres $-h'(t)$:
\begin{equation}\label{def_drag}
\F=-\kappa_{\mathrm{pass}}h'(t).
\end{equation}

The drag coefficient $\kappa_{\mathrm{pass}}$ is a function of half-distance $h$. 
It also depends on the parameter $\beta\geq 0$, i.e., it depends on the choice of the boundary conditions.

The dependence of $\kappa_{\mathrm{pass}}$ on $h$ for no-slip boundary conditions, $\beta=0$, is known as a classical result (see, e.g., \cite {CoxBre67,CooNei37}). 
For Navier boundary conditions, $\beta>0$, the result has been justified recently 
(see \cite {GerHil11}):
\begin{eqnarray}
    &&
    \text{if } 0<h<\beta<<1,
    \text{ then } \kappa_{\mathrm{pass}}\asymp \frac{1}{\beta}\ln\frac{1}{h},
    \label{k_navier}\\
    &&
    \text{if } 0 \leq \beta<h<<1,
    \text{ then } \kappa_{\mathrm{pass}}\asymp \frac{1}{h}.
    \label{k_pass}
\end{eqnarray}

%Consider the case of no-slip boundary conditions and the case of Navier boundary conditions separately.

{First, consider no-slip boundary conditions.}
From \eqref{k_pass} with $\beta=0$ we may conclude that for some positive constant $C>0$ the following inequality holds:
\begin{equation}\label{Cdh}
    \kappa_{\mathrm{pass}}< \frac{C}{h}.
\end{equation}
Using \eqref{be}, \eqref{def_drag}, and \eqref{Cdh}, we obtain
\begin{equation*}
    h(t)>h(0)\mathrm{exp}(-Cf_{\mathrm{ext}} t).
\end{equation*}

It implies that $h(t)$ can not vanish for any finite $t>0$.
It says that the distance between spheres will never become zero,
thus they will never collide whatever positive numbers $h(0)$ and $f_{\mathrm{ext}}$ are.
This contradicts to what happens in real life and it is known like {\it no-collision paradox}.

{\it Secondly, consider Navier boundary conditions.}
It is sufficient to restrict ourselves to the case $h<\beta<<1$. The relation \eqref{k_navier} implies that for some positive constant $C>0$ the following inequality holds:
\begin{equation}\label{Clnh}
    \kappa_{\mathrm{\text{pass}}}>\frac{C}{\beta}\ln\frac{1}{h}.
\end{equation}
Using \eqref{be}, \eqref{def_drag} and \eqref{Clnh}, we obtain
\begin{equation*}
    Ch(t)(\ln h(t)-1)-Ch(0)(\ln h(0)-1)>f_{\mathrm{ext}}t.
\end{equation*}
This inequality together with $h'(t)<0$ predicts that there exists $T_{\text{coll}}$ such that $0<T_{\text{coll}}<\infty$ and $h(T_{\text{coll}})=0$. Thus, in the case of Navier boundary conditions collisions do occur.

%%%%%%%%%%%%%Collision time %%%%%%%%%%%%%%%%%%%%%%%%%%%%%%%
%Time before collision $T_{\text{coll}}$ can be evaluated by the following formula
%\begin{equation}
%    T_{\text{coll}}=\int_{0}^{h(0)}\frac{dh}{U(h)},
%\end{equation}
%where $U(h)=-h'(t)$.
%The following estimate for the collision time holds ($f_{\text{ext}}=1,\;h(0)=1$):
%\begin{equation}
%T_{\text{coll}}\sim 1/\beta \ln (1/h).
%\end{equation}

%\begin{eqnarray*}
%    T_{\text{coll}}
%    &=&
%    \int_{0}^{\beta}\frac{dh}{U(h)}+\int_{\beta}^{1}\frac{dh}{U(h)}
%    =
%    \int_{0}^{\beta}\kappa_{\text{pass}}(h)dh+\int_{\beta}^{1}\kappa_{\text{pass}}(h)dh\\
%    &<&
%    \frac{1}{c_0\beta}\int_{0}^{\beta}\ln (1/h)dh
%    +
%    \frac{C}{\beta}\int _{0}^{\beta}dh+C\int _{0}^{\beta}\ln (1/h)dh+c\int_{\beta}^{1}\frac{dh}{h} \\ \\
%&<& c_1\ln (1/\beta)+c_2.
%\end{eqnarray*}
%where $c_1$ and $c_2$ are some positive constant independent from $\beta$.
%%%%%%%%%%%%%%%%%%%%%%%%%%%%%%%%%%%%%%%%%%%%%%%%%%%

%...Lorentz Reciprocal Theorem+ Exact Solution....
\section {Two main results}
\label{sect:main results}

\subsection{No collisions for swimmers with no-slip boundary conditions}
\label{sect:results for noslip}

The following theorem is the main result for the problem with the no-slip boundary conditions.

\begin{theorem} Consider coupled PDE/ODE model \eqref{stokes},\eqref{noslip} and \eqref{balance}. Let $h(t)$ be a half-distance between two swimmers at  $t>0$, introduced in Section 3. Then there exist positive constants $C_1>0$ and $C_2>0$ such that
\begin{equation}\label{main}
    h(t)>C_1e^{-C_2t}
\end{equation}
for all $t>0$.
\end{theorem}

\noindent{\it P r o o f.} For simplicity, consider the case without inertia ($m=0$). The case $m>0$ relies on the same arguments.
Let $u$ satisfy \eqref{stokes},\eqref{noslip} and \eqref{balance}.
Then $u=v+w$ where $v$ and $w$ satisfy the following systems of equations:\\
\centerline{
I.:
$\left\{
\begin{array}{l}
\Delta v-\nabla p_v=0,\\
\nabla\cdot v=0,\\
v=V\d^i,\\
\F(v)=f_p\d^i,
\end{array}
\right.
$
II.:
$\left\{
\begin{array}{l}
\Delta w-\nabla p_w=\sum\limits_{i=1,2}\delta(\x-\x_i^p)f_p\d^i,\\
\nabla\cdot w=0,\\
w=W\d^i,\\
\F(w)=0.
\end{array}
\right.
$
}
In both systems, I and II, equalities in the first and the second line are satisfied in $\Omega_h$, equalities in the third line are satisfied on $\partial B_i$ for all $i=1,2$, equalities in the fourth line are satisfied for all $i=1,2$.
If $f_p$ and $h$ are given we can find $v$ and $q$ from the first, the second and the fourth equalities of system I.
After that we can find $V\in\mathbb R$ from the third equality.
Similarly, for system II. Thus, there exist mappings $G_{v}$ and $G_{w}$ defined by problem I and problem II, respectively, such that $\mathcal{G}_{v}(h,f_p)=V$ and $\mathcal{G}_{w}(h,f_p)=W$ (like the Stokes law for a sphere $\F=6\pi\mu U$).

The result follows from the following two propositions:
\begin{proposition}\label{prop1}
Let $V\in\mathbb R$ be the number obtained from system I (i.e., $V=\mathcal{G}_v(h,f_p)$). Then there exists a positive constant $C>0$ such that
\begin{equation}\label{V}
V< Cf_ph.
\end{equation}
\end{proposition}

The result of proposition \ref{prop1} follows from the relations \eqref{Cdh} (see \cite{CoxBre67,CooNei37}) and
\begin{equation*}
f_p=V\kappa_{\text{pass}}.
\end{equation*}

\begin{proposition}\label{prop2}
Let $W\in\mathbb R$ be the number obtained in system II (i.e., $W=\mathcal{G}_w(h,f_p)$). Then the following inequality holds:
\begin{equation}\label{W}
W< 0.
\end{equation}
\end{proposition}

The proof of proposition \ref{prop2} is relagated to Section \ref{sect:proof of the noslip result}.

Now \eqref{V}, \eqref{W} and $u=v+w$ imply that $-h'(t)=V+W<Cf_ph(t)$. Substituting it to \eqref{balance} we have
$$
h'(t)> -Cf_ph(t),
$$
which implies the statement of the theorem. $\square$

\subsection{Finite time collisions for swimmers with Navier boundary conditions}
\label{sect:result for Navier}

%Now we present the result for Navier boundary conditions.
\begin{theorem} \label{thm_navier} Consider coupled PDE/ODE model \eqref{stokes},\eqref{navier}, \eqref{balance} and $m>0$. Then there exists a positive constant $C>0$ which may depend on $h(0)$, $m$, $\beta$ and $\lambda$ such that if $h'(0)>C$, then there exits time $T_{\text{coll}}>0$ such that $h(T_{\text{coll}})=0$.
\end{theorem}
$T_{\text{coll}}$ is time of collision. Quantative etimates on $T_{\text{coll}}$ which can be compared with measurements have been obtained in \cite{PotGyrAraBer13}
(see also Remark \ref{remark: colltime}). The proof of Theorem \ref{thm_navier} is relagated to Subsection \ref{proof of navier}. The proof is based on variational principle and the Poincare type inequality (see Subsection \ref{auxiliary result}).

\begin{remark} The Theorem \ref{thm_navier} states that swimmers collide if the initial speed is sufficiently large.
Thus, the result on passive spheres is extended on active swimmers.
But one can expect that swimmers must collide for arbitrary initial speed and separation distance. This is due to the propulsion force pushing swimmers towards each other which is absent in the case of passive spheres.
In order to prove collisions without any restrictions on initial conditions and $m\geq 0$ one need to show
\begin{equation}\label{prop_less_than_one}
1-\kappa_{\text{prop}}>l(\lambda,\beta),
\end{equation}
where $\kappa_{\text{prop}}$ was introduced in Section 3 and represents the drag force on a swimmer generated by self-propulsion $f_p=1$ of both swimmers. The function $l(\lambda,\beta)$ is a strictly positive function, it may depend on $\lambda$ and $\beta$, but does not depend on the swimmer's separation $h$.
The inequality \eqref{prop_less_than_one} means that the drag force due to propulsion mechanism ($\delta$-functions in the Stokes equation) is not greater than the actual propulsion.
In other words, the pushing mechanism really pushes swimmers towards each other.

\begin{figure}\label{Fig 2}
\centerline{\includegraphics[height=3 cm]{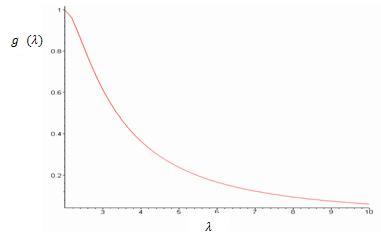}}
\caption{Dependence of $\kappa_{\text{prop}}$ on flagella length $\lambda$ with $h=0.01$ and $\beta=0.1$}
\end{figure}
For fixed roughness parameter $\beta=0.1$ and small separation distance $h=0.01$ the force $\kappa_{\text{prop}}$ is a function of tail length $\lambda$, denote it by $g(\lambda)$.
Numerics suggest that inequality \eqref{prop_less_than_one} holds for all $\lambda,\beta>0$. It is optimal in the following sense: when $\lambda\rightarrow 0$ the function $g(\lambda)=\kappa_{\text{prop}}\to  1$ (see Fig 2). % \ref{Fig 2}).
This makes the inequality \eqref{prop_less_than_one} difficult to prove.
\end{remark}
\begin{remark}\label{remark: colltime}
It is also interesting to estimate the time before collision $T_{\text{coll}}$
in terms of parameter $\beta$.
Time before collision $T_{\text{coll}}$ can be evaluated by the following formula
\begin{equation}\label{t_coll}
    T_{\text{coll}}=\int_{0}^{h(0)}\frac{dh}{U(h)},
\end{equation}
where $U(h)=-h'(t)$.
The following estimate for the collision time holds ($f_{\text{ext}}=1,\;h(0)=1$):
\begin{equation}
T_{\text{coll}}\sim 1/\beta \ln (1/h).
\end{equation}
To estiblish the formula \eqref{t_coll}, one can use \eqref{expl_be}, the assumption \eqref{prop_less_than_one} and the relation for \eqref{k_navier}. Numerics on collision time are performed in our paper \cite{PotGyrAraBer13}.
\end{remark}
\section{Proof of proposition \ref{prop2}}
\label{sect:proof of the noslip result}

\noindent We use the following well-known theorem
\begin{theorem}[Lorentz Reciprocal Theorem% or 2nd Green Formula for Stokes Problem
] \label{thm: LRT}Assume that equalities $\Delta u^1-\nabla p^1=-F_1$, $\Delta u^2-\nabla p^2=-F_2$ and $\nabla\cdot u^1=0$, $\nabla \cdot u^2=0$ hold in domain $\Omega$. Then
\begin{equation}\label{LRT}
\int\limits_{\Omega}F_1\cdot u^2d\x-\int\limits_{\partial\Omega}\sigma(u^1,p^1)\n\cdot u^2ds=
\int\limits_{\Omega}F_2\cdot u^1d\x-\int\limits_{\partial\Omega}\sigma(u^2,p^2)\n\cdot u^1ds.
\end{equation}
Here $\n$ is an inward (for $\Omega$) normal vector.
\end{theorem}

Let us take $\Omega=\Omega_h$, $u^2=w$, $p^2=p_w$ and $F_2=-\sum\limits_{i=1,2}\delta(\x-\x_p^i)f_p\d^i$. Let $u^1$ be such function that $\Delta u^1-\nabla p^1=0$ (in other words, $F_1=0$), and $u^1$ satisfies the following boundary conditions:
$$
    u^1=W\d^i  \text{ on } \partial B_i,\;\;i=1,2.
$$
From \eqref{LRT} we get:
$$
    W\sum\limits_{i=1,2}\d_i\cdot\int\limits_{\partial B_i}\sigma(u^1,p^1)\n ds=
    f_p\sum\limits_{i=1,2}\d^i\cdot u^1(\x_p^i).
$$
The following equality holds:
$$
    \int\limits_{\partial B_i}\sigma(u^1,p^1)\n = \F(u^1)\d^i,
$$
and due to symmetry of the problem for $(u^1,p^1)$ we have
$$
    \d^1\cdot u^1(\x_p^1)=\d^2\cdot u^1(\x_p^2)=-u^1_z(\x_p^1), \;\;i=1,2,
$$
where $u^1_z(\x_p^1)$ is a $z$-component of $u^1(\x_p^1)$.
Thus, we get
\begin{equation} \label{important}
    W=-\frac{f_p}{\F(u^1)}u^1_z(\x_p^1).
\end{equation}
%Once the following lemma is proven, the proof of proposition \ref{prop2} is complete.
Thus, it remains to prove positivity of $u_{z}(\x_p^1)$.
\begin{lemma}
The following inequality holds
\begin{equation}\label{posit}
    u^1_z(\x_p^1)>0.
\end{equation}
\end{lemma}

\noindent
{\it P r o o f.} We recall that $u^1$ satisfies the following system
\begin{eqnarray}
\nonumber &&
\left\{\begin{array}{l}
\Delta u^1 -\nabla p^1 =0, \text{  in  }\Omega_{h}\\
\nabla\cdot u^1=0, \text{  in  }\Omega_{h}
\end{array}\right.\\
&&\nonumber
u^1 = W\d^i, \text{  on  }\partial B_i,\;\;i=1,2.
\end{eqnarray}
The proof is divided into three parts. First, the result on p. 247 is \cite{Bre61} is used to write the exact solution of system above. Second, we calculate $u^1$ at a point $\x_p^1$, and, finally, a necessary estimate is obtained.
\vskip 2mm
\noindent{STEP 1. }{\it Exact solution.} Introduce cylindrical coordinates for $\x=(x,y,z)$
$$
\begin{array}{l}
    x=\rho\cos\phi,\\y=\rho\sin\phi,\\z=z.
\end{array}
$$
Denote  $e_{\rho}=(\cos\phi,\sin\phi,0)$ and $e_z=(0,0,1)$. Due to the axial symmetry of the problem we have $u^1=u_{\rho}e_{\rho}+u_{z}e_z$ where scalar functions $u_{\rho}$ and $u_z$ do not depend on $\phi$.
Following \cite{Bre61} we introduce the stream function $\psi=\psi(\rho,z)$
$$
    u_z(z,\rho)
    =
    -\frac{1}{\rho}\frac{\partial}{\partial \rho}\psi(z,\rho),\;\;\;u_{\rho}(z,\rho)=\frac{1}{\rho}\frac{\partial}{\partial z}\psi (z,\rho).
$$

In order to write function $\psi$ we introduce the bipolar coordinates $\zeta\in[0,+\infty)$, $\mu\in[0,\pi]$ (we need the case $z>0$ only)
$$
z=c\frac{\sinh\zeta}{\cosh\zeta-\cos\eta},\;\;\rho=c\frac{\sin \eta}{\cosh\zeta-\cos\eta},
$$
where $c=\sinh\alpha$ and $\alpha$ are such positive number that $\cosh \alpha =1+h$.
Note that due to the following equality
$$
(z-c\coth \zeta)^2+\rho^2=(c \;\mathrm{cosech}\zeta)^2
$$ one can easily verify that surface $\left\{\zeta=\alpha\right\}$ is the sphere $\partial B_1$. Also we will need the formula:
\begin{equation}\label{ln}
\zeta+i\eta=\ln\left(\frac{\rho+i(z+c)}{\rho+i(z-c)}\right).
\end{equation}

The function $\psi$ may be written as follows (see \cite{StiJef26,Bre61})
$$
\psi(\zeta,\eta)=(\cosh\zeta-\cos\eta)^{-3/2}\sum\limits_{n=0}^{\infty}U_n(\zeta)C_{n+1}^{-1/2}(\cos\eta),
$$
where
\begin{eqnarray}
U_n(\zeta)&=&b_n\sinh(n-\frac{1}{2})\zeta+d_n\cosh(n+\frac{3}{2})\zeta,\nonumber\\
C_{n+1}^{-1/2}(x)&=&\frac{P_{n-1}(x)-P_{n+1}(x)}{2n+1}.\nonumber
\end{eqnarray}
Here $P_n(x)$ are the Legendre polynomials
$$
P_n(x)=\frac{1}{2^nn!}\frac{d^n}{dx^n}(x^2-1)^{n}
$$
and the coefficients $b_n$ and $d_n$ are given by the following formulas
\begin{eqnarray}
\nonumber b_n&=&WR^2\sinh^2\alpha\frac{n(n+1)}{\sqrt{2}(2n-1)}\times\\&&\times
\nonumber\left[
\frac{4\cosh^2(n+\frac{1}{2})\alpha+2(2n+1)\sinh^2\alpha}{2\sinh(2n+1)\alpha-(2n+1)\sinh2\alpha}
-1\right],\\
\nonumber d_n&=&WR^2\sinh^2\alpha\frac{n(n+1)}{\sqrt{2}(2n+3)}\times\\&&\times
\nonumber\left[1-
\frac{4\cosh^2(n+\frac{1}{2})\alpha-2(2n+1)\sinh^2\alpha}{2\sinh(2n+1)\alpha-(2n+1)\sinh2\alpha}
\right].
\end{eqnarray}
\vskip 2mm
\noindent{STEP 2.}{ \it The calculation of $u^1(\x_p^1)$.} We need to calculate $u^1$ when $z_0=1+h+\xi$ and $\rho\rightarrow 0+$. In bipolar coordinates it corresponds to
$$
\zeta_0= \ln \frac{z_0+c}{z_0-c} \;\;\text{ and }\;\;\eta_0=0.
$$

Then using \eqref{ln} (to obtain $\frac{\partial \zeta}{\partial \rho}$ and $\frac{\partial \eta}{\partial \rho}$), $P_n(1)=1$, $P'_n(1)=n(n+1)/2$ and the equality
$$
\frac{\partial \psi}{\rho\partial\rho}=\frac{\partial \zeta}{\rho\partial\rho}
\frac{\partial\psi}{\partial \zeta}+\frac{\partial \eta}{\partial\rho}\frac{\partial\psi}{\rho\partial \eta}
$$
we obtain
\begin{eqnarray}\nonumber
u_z(z_0,0)&=&-\lim\limits_{\rho\rightarrow 0+}\frac{1}{\rho}\frac{\partial}{\partial \rho} \psi(z_0,\rho)\\&=&-\frac{2}{c^2-z_0^2}(\cosh^{-1}\zeta_0-1)^{-1/2}\sum\limits_{n=0}^{\infty}U_n(\zeta_0).\label{u_z}
\end{eqnarray}

Since $c^2=h(2+h)$ we have that $c^2-z_0^2<0$ and the statement of proposition follows if we prove that $U_n(\zeta_0)>0$ for all $n\in\mathbb N$.
\vskip 2mm

\noindent{STEP 3. }{\it Positivity of $U_n(\zeta_0)$.} The nonpenetration condition $(u-\d^i)\cdot n=0$ implies the following relation between coefficients $b_n$ and $d_n$ (see, e.g., formula (42) in \cite{ReeMor74}):
\begin{equation}\label{nonpenetration}
b_n =G_m-d_n\frac{\sinh(m+1)\alpha}{\sinh(m-1)\alpha},
\end{equation}
where $m=n+1/2$ and
$$
G_m=\frac{c^2}{2\sqrt{2}}\frac{m^2-1/4}{m^2-1}\left[\frac{(m+1)e^{-(m-1)\alpha}-(m-1)e^{-(m+1)\alpha}}{\sinh(m-1)\alpha}\right].
$$
The formula for $U_n(\xi)$ has the following form:
\begin{equation}\label{U_m_in}
U_n(\xi)=\left[G_m+d_n\left\{\frac{\sinh(m+1)\xi}{\sinh (m-1)\xi}-\frac{\sinh(m+1)\alpha}{\sinh (m-1)\alpha}\right\}\right]\sinh(m-1)\xi.
\end{equation}

\noindent Note that
\begin{equation*}
(m+1)e^{-(m-1)\alpha}-(m-1)e^{-(m+1)\alpha}=2e^{-m\alpha}(m\sinh \alpha+\cosh \alpha)>0
\end{equation*}
Thus,
\begin{equation}\label{G_m}
G_m>0.
\end{equation}

\noindent Introduce $f(\xi):=\sinh(m+1)\xi\sinh(m-1)\alpha-\sinh(m-1)\xi\sinh (m+1)\alpha$. Consider $\xi\in[0,\alpha]$. One can see that $f(0)=f(\alpha)=0$. Consider now $M:=\max\limits_{\xi\in[0,\alpha]}f(\xi)=f(\xi_0)$. Assume that $M>0$. Then \begin{eqnarray}
f''(\xi_0)&=&(m+1)^2\sinh(m+1)\xi\sinh(m-1)\alpha-(m-1)^2\sinh(m-1)\xi\sinh (m+1)\alpha\nonumber \\&&\nonumber >(m-1)^2f(\xi_0)>0,
\end{eqnarray}
what is impossiple for maximum. It means that $f(\xi)\leq 0$ for $\xi\in (0,\alpha)$. Thus,
\begin{equation}\label{the_second_term}
\frac{\sinh(m+1)\xi}{\sinh (m-1)\xi}-\frac{\sinh(m+1)\alpha}{\sinh (m-1)\alpha}<0.
\end{equation}
Combining \eqref{G_m} and \eqref{the_second_term} we obtain from \eqref{U_m_in} that $U_n(\xi_0)>0$.

\section{Proof of the collision result for Navier BC}
\label{sect:proof of the Navier result}

\subsection{Auxiliary statements}\label{auxiliary result}

Consider the following Stokes problem for smooth force density ${\bf F}$:
\begin{equation}\label{stokes_aux_th2}
\left\{\begin{array}{l}
-\Delta u+\nabla p={\bf F},\\
\nabla \cdot u =0,\\
(u-\d^i)\cdot \n=0,\\
(u-\d^i)\times \n=-2\beta\left[D(u)\n\right]\times \n .
\end{array}\right.
\end{equation}
Define energy functional:
\begin{equation}
\mathcal{E}_{\bf F}(v)=2\int|D(v)|^2+\frac{1}{\beta}\sum_{i=1,2}\int_{\partial B_i}|v-\d^i|^2-\int {\bf F}v
\end{equation}
defined on the following class of functions:
\begin{equation*}
\mathcal{A}=\left\{v| \;(v-\d^i)\cdot \n =0 \text{ on }\partial B^i\right\}.
\end{equation*}
One can easily prove the following statement:
\begin{proposition}\label{prop: variation principle}
Let $u$ solve Stokes problem \eqref{stokes_aux_th2}. Then the following statements hold true:
\begin{list}{}{}
\item{(i)} $u$ is a minimizant of $\mathcal{E}_{\bf F}$ on class $\mathcal{A}$:
\begin{equation*}
\min\limits_{v\in\mathcal{A}}\mathcal{E}_{\bf F}(v)=\mathcal{E}_{\bf F}(u).
\end{equation*}
\item{(ii)} Drag force $\mathcal{F}(u)=\int_{\partial B^i}\sigma(u,p)\n\cdot \d^i$ may be evaluated by the following formula:
\begin{equation*}
\mathcal{F}(u)=\mathcal{E}_{\bf F}(u)+\int {\bf F} u.
\end{equation*}
\end{list}
\end{proposition}
We will also need the following Poincare type inequality:
\begin{proposition}\label{prop: poincare special}{\it
Let $v$ be the function such that $v\cdot \n=0$ on $\partial B_i$. Then
 \begin{equation}\label{fromkorn}
\int\limits_{\Omega^*}|v|^2< C \int\limits_{\tilde{\Omega}}|D(v)|+\int\limits_{\Gamma}|v|^2,
\end{equation}
where $\Omega^{*}$ and $\tilde{\Omega}$ are bounded domains such that $\Omega^*\subset\tilde\Omega\subset  \Omega_h$, and the set $\Gamma:=\partial \tilde{\Omega}\cap \partial B_i$ is not empty  (see Fig. 3)}.% \ref{Fig 3})}.
\begin{figure}\label{Fig 3}
\centerline{\includegraphics[ height=4 cm]{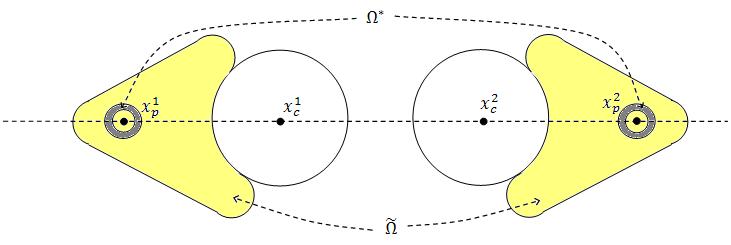}}
\caption{Domains $\Omega^*$ and $\tilde{\Omega}$.}
\end{figure}
\end{proposition}

\begin{proof}{$\;$}\\
For the sake of simplicity we consider the case when $\tilde\Omega$ and $\Omega^*$ are symmetric with respect to plane $xy$: if $\x=(x,y,x)\in \Omega^{*}(\text{or }\tilde{\Omega})$, then $\x^{\text{sym}}:=(x,y,-z)\in \Omega^{*}(\text{or }\tilde{\Omega})$.\\
\noindent{STEP 1.} Assume by contradiction that there exists such sequence $$\left\{v_n: v_n\cdot \n=0\text{ on $\Gamma $}\right\}$$ that
\begin{equation}
\int_{\Omega_*}|v_n|^2=1,\int_{\tilde{\Omega}} |D(v_n)|\rightarrow 0 \text{ and }\int\limits_{\Gamma}|v_n|^2\rightarrow 0.
\end{equation}
Let $V=\left\{\left.\gamma+\Lambda\times \x\right| \;\gamma,\Lambda \in \mathbb R^3\right\}$ be the linear subspace of $\left[L^2(\tilde{\Omega})\right]^3$ ("rigid motion space") and $\Pi:\left[L^2(\tilde{\Omega})\right]^3\rightarrow V$ is the corresponding projection operator.
Introduce $w_n:= (I-\Pi)v_n$. Then for all $n$ there exit $\gamma,\Lambda \in \mathbb R^3$ such that
\begin{equation}
v_n(\x)=\gamma_n+\Lambda_n\times \x +w_n(\x).
\end{equation}
\noindent{STEP 2.} We claim that
\begin{equation}\label{claim2}
\|w_n\|^2_{H^1(\tilde{\Omega})}=\int _{\tilde{\Omega}}|w_n|^2+|\nabla(w_n)|^2\rightarrow 0.
\end{equation}
Indeed, the functions $w_n$ satisfy the Korn inequality (see, e.g., Theorem 2.5 in \cite{OleShaYos92}):
\begin{equation}\label{w_korn}
\int_{\tilde{\Omega}}|\nabla w_n|^2\leq C\int_{\tilde{\Omega}} |D(w_n)|^2.
\end{equation}
On the other hand, functions $w_n$ are orthogonal in $\left[L^2(\tilde{\Omega})\right]^3$ to $V$ and, thus, they are orthogonal to constant three-dimensional vectors. It means that the avarage of each component of $w_n$ in $\tilde{\Omega}$ is zero and Poincaret inequality implies that
\begin{equation}\label{w_poincaret}
\int_{\tilde{\Omega}}|w_n|^2<C\int_{\tilde{\Omega}}|\nabla(w_n)|^2
\end{equation}
Collecting togeather \eqref{w_korn}, \eqref{w_poincaret} and $D(w_n)=D(v_n)$ we get  \eqref{claim2}.\\
\noindent STEP 3. We claim that there exists a sequence $\left\{n_k\subset \mathbb N\right\}_{k=1}^{+\infty}$
\begin{equation}
\gamma_{n_k}\rightarrow 0 \text{ and }\Lambda_{n_k}\rightarrow 0\text{ as }k\rightarrow+\infty.
\end{equation}
Indeed, $V$ is a finite dimensional space equipped with $L^2$-norm. The sequence $$\left\{\gamma_n+\Lambda_n\times\x=v_n-w_n\right\}$$ is bounded in $\left[L^2(\tilde{\Omega})\right]^3$ and, thus, there exists a convergent subsequence $$\gamma_{n_k}+\Lambda_{n_k}\times \x\rightarrow \gamma + \Lambda \times \x.$$
From $\int_{\Gamma}|v_n|^2\rightarrow 0$ and that $\Gamma$ contains peace of the sphere $\partial B_i$ we get
\begin{equation}
\int_{\Gamma}|\gamma +\Lambda \times \x|^2=0\Rightarrow \gamma +\Lambda \times\x=0\text{ for all $\x\in \Gamma$}
\end{equation}
and since span of all vectors $\x\in \Gamma$ coincides with $\mathbb R^3$ we get that $\gamma$ and $\Lambda$ are zero vectors. This contradicts to the relation:
\begin{equation}
\int\limits_{\tilde{\Omega}} |v_n|^2=1.
\end{equation}
 %\rightline{$\square$}
\end{proof}

\begin{remark} \rm The inequality \eqref{fromkorn} without the integral over $\Gamma$
 \begin{equation}\label{fromkorn 1}
\int\limits_{\Omega^*}|v|^2< C \int\limits_{\tilde{\Omega}}|D(v)|
\end{equation}
is not valid for  function $v$ such that $v\cdot \n =0$ on $\Gamma$. One can take nonzero vector $\Lambda\in \mathbb R^3\backslash \left\{0\right\}$ and the sequence:
$$
v_n=n\Lambda\times (\x-\x_c^1)
$$
Then $v_n\cdot \n=v_n\cdot(\x-\x_c^1)=n\Lambda\times (\x-\x_c^1)\cdot(\x-\x_c^1)=0$, $\int_{\Omega^*} |v_n|^2=n^2\int_{\Omega^*}|\Lambda\times(\x-\x_c)|^2\rightarrow +\infty$, while $|D(v_n)|=0$.
\end{remark}
\begin{corollary} For functions $v$: $(v-\d^i)\cdot \n=0$ the following holds
\begin{equation}\label{now}
\int\limits_{\Omega^*}|v|^2< C \int\limits_{\tilde{\Omega}}|D(v)|+C\int\limits_{\Gamma}|v-\d^i|^2+C<C\mathcal{E}(v)+C.
\end{equation}
\end{corollary}

\subsection{Proof of Theorem \ref{thm_navier}}
\label{proof of navier}
We consider the following Stokes problem \begin{equation*}
\left\{\begin{array}{l}
-\Delta u+\nabla p=\sum_{i=1,2}f_p\d^i\delta(\x-\x_p^i),\\
\nabla \cdot u =0,\\
(u+h'(t)\d^i)\cdot \n=0,\\
(u+h'(t)\d^i)\times \n=-2\beta\left[D(u)\n\right]\times \n ,
\end{array}\right.
\end{equation*}
and the balance equation
\begin{equation}
mh''(t)-\F(u)+f_p=0.
\end{equation}
where $\F(u)=\d^i\cdot\int _{\partial B_i}2D(u)\n-p\n$. The definition of $\F$ does not depend on $i=1,2$.

\noindent{STEP 1.} First, we reduce the original system with singular force density to the problem with regular one.

Denote $K:=-1/f_ph'(t)$. Consider the function $\tilde{u}_{\Phi}$ defined as follows:
\begin{equation*}
\tilde{u}_{\Phi}=-\frac{1}{r}\cdot\frac{\partial}{\partial z} \left[\chi \psi_{\Phi}\right] \e_r+\frac{1}{r}\cdot \frac{\partial}{\partial r}\left[\chi\psi_\Phi\right] \e_z.
\end{equation*}
where $\chi:\mathbb R^3\rightarrow [0,1]$ is $C^{\infty}$-function such that
\begin{equation*}
\chi(\x)=\left\{\begin{array}{ll}1,&\x\in \mathcal{U}_{\lambda/3}(\x_p^i),\\ 0,&\x\in \mathcal{U}_{\lambda/2}(\x_p^i).\end{array}\right.
\end{equation*}
and $\psi_\Phi(r,z):\mathbb R_r\times \mathbb R_z\rightarrow \mathbb R$:
\begin{equation*}
\psi_{\Phi}=\frac{K}{8\pi }\left[\frac{1}{\sqrt{(z-l_p)^2+r^2}}-\frac{1}{\sqrt{(z+l_p)^2+r^2}}\right],
\end{equation*}
where $l_p=h+2+\lambda$ is a $z-$coordinate of a tail.
Function $\tilde{u}_{\Phi}$ solves the following problem:
\begin{equation}
\left\{\begin{array}{l}
-\Delta \tilde{u}_{\Phi}+\nabla \tilde{p}_{\Phi}=\sum\limits_{i=1,2}K \d^i\delta(\x-\x_i) -F(\x),\\
\nabla\cdot u =0,\\
\tilde{u}_{\Phi}\cdot \n=0,\\
\tilde{u}_{\Phi}\times \n= -2\beta\left[D(\tilde{u}_{\Phi})\n\right]\times \n.\end{array}
\right.
\end{equation}
for some scalar $\tilde{p}_{\Phi}$ and $C^{\infty}$ function $F$ which support belongs to $\Omega^*=\left\{\xi/3<|\x-\x_i|<\xi/2,i=1,2\right\}$.
Consider $\overline{u}=-(1/K)u-\tilde{u}_{\Phi}$. It satisfies the following problem
\begin{equation}
\left\{\begin{array}{l}
-\Delta \overline{u}+\nabla {\overline{p}}=F(\x),\\
\nabla\cdot \overline{u} =0,\\
(\overline{u}-\d^i)\cdot \n=0,\\
(\overline{u}-\d^i)\times \n=-2\beta\left[D(\overline{u})\n\right]\times \n.
\end{array}
\right.
\end{equation}

\noindent{STEP 2.} We will prove the following inequality:
\begin{equation}\label{we_will_prove}
\F(\overline{u})<C(|\ln h|+K^2).
\end{equation}
Assume for simplicity $K>1$.\\
1. We claim that
\begin{equation}
\left|\int F\overline{u}\right|<\frac{1}{4}\mathcal{E}(\overline{u})+C.
\end{equation}
where
\begin{equation}
\mathcal{E}(v)=2\int|D(v)|^2+\frac{1}{\beta}\sum_{i=1,2}\int_{\partial B_i}|v-\d^i|^2
\end{equation}
Indeed, in view of \eqref{now}
\begin{eqnarray}
\left|\int F\overline{u}\right|<KC\left(\int\limits_{\Omega^*} |\overline{u}|^2\right)^{1/2}\leq KC \sqrt{\mathcal{E}(\overline{u})+1}<\frac{1}{4}\mathcal{E}(u)+CK^2.
\end{eqnarray}
2. Consider $\mathcal{E}_F(v)$ defined as follows
\begin{equation}
\mathcal{E}_F(v)=2\int|D(v)|^2+\frac{1}{\beta}\sum_{i=1,2}\int_{\partial B_i}|v-\d^i|^2-\int Fv
\end{equation}
Then from 1. we have that
\begin{equation}
\mathcal{E}_F(\overline{u})>\frac{1}{2}\mathcal{E}(\overline{u})-CK^2
\end{equation}
Observe that
\begin{eqnarray}
\F(\overline{u})&=&\mathcal{E}_F(\overline{u})+\int F \overline{u}<\mathcal{E}_F(\overline{u})+\frac{1}{4}\mathcal{E}(\overline{u})+CK^2\\
&&\leq C\mathcal{E}_F(\overline{u})+CK^2<C\min\limits_{(v-\d^i)\cdot n =0}\mathcal{E}(v)+CK^2\\
&&\leq C|\ln h|+CK^2.
\end{eqnarray}
And thus inequality \eqref{we_will_prove} is proven.

\noindent{STEP 3.}
The balance equation is
\begin{equation}
mh''(t)-\F(u)+f_p=0.
\end{equation}
The definition of $\F$ does not depend on $i=1,2$. \\
Due to Step 2 the following inequality holds:
\begin{equation}
mh''(t)-Ch'(t)\ln h(t)+C/h'(t)+f_p<0
\end{equation}
Consider any $\nu>0$. Assume that $h'(t)<-\nu$ for all $t\in(0,t_{\star})$ and $h'(t_{\star})=-\nu$. Then
\begin{equation}
mh''(t)-Ch'(t)\ln h(t)<-f_p-C/h'(t)<-f_p+C/\nu
\end{equation}
Then by the integration
\begin{eqnarray}
mh'(t)&<&(-f_p+C/\nu)t+mh'(0)\nonumber\\&&+C\left\{h(t)(\ln h(t)- h(t))-h(0)(\ln h(0)- h(0))\right\}\nonumber\\
&<&(-f_p+C/\nu)t+mh'(0).\nonumber
\end{eqnarray}
Let us take $-mh'(0)>C/\nu T+m(\nu+h(0)/T)$ then
\begin{equation}
mh'(t)<-f_pt-C/\nu (T-t)-m\nu-mh(0)/T.
\end{equation}
Thus $t_{\star}>T$ and $-h'(0)<-h(0)/T$. Thus, time of collision $T_{\text{coll}}$ is less then $T$.\\
This completes the proof of theorem \ref{thm_navier}.

\section*{Acknowledgment}
The authors would like to thank Dr. Volodymyr Rybalko for many useful suggestions and I. Aranson for very useful discussion on biophysical aspect of the problem. \\
The work of Leonid Berlyand and Mykhailo Potomkin was supported by the DOE Grant No. DE-FG02-08ER25862.
Vitaliy Gyrya was supported by DOE Grant No. DE-AC52-06NA25396.

%%%%%%%%%%%%%%%%%%%%%%%%%%%%%%%%%%%%%%%%%%%%%%%%%%%%%%%%%%%%%%%%%%%%

%\bibliographystyle{ieeetr}
\bibliographystyle{ieeetr}
\bibliography{collisions}

\end{document}